\documentclass{amsart}

\usepackage{xspace, graphicx, epsfig, amssymb, amsmath, amsthm, stmaryrd}

\newtheorem{theorem}{Theorem}[section]
\newtheorem{lemma}[theorem]{Lemma}

\theoremstyle{definition}
\newtheorem{definition}[theorem]{Definition}

\newtheorem{prop}[theorem]{Proposition}
\newtheorem{cor}[theorem]{Corollary}

\theoremstyle{remark}

\numberwithin{equation}{section}

\newcommand {\Z} {\mathbb{Z}}

\newcommand {\C} {\mathbb{C}}

\newcommand {\End} {\mathrm{End}\:}

\newcommand {\St} {\mathrm{St}}

\def\gl{\mathfrak {gl}}

\def\g{\mathfrak {g}}
\def\l{\mathfrak {l}}

\def\n{\mathfrak {n}}
\def\p{\mathfrak {p}}
\def\sl{\mathfrak{sl}}
\def\so{\mathfrak{so}}
\def\sp{\mathfrak{sp}}

\def\k{\mathfrak{k}}
\def\q{\mathfrak{q}}
\def\r{\mathfrak{r}}
\def\s{\mathfrak{s}}
\def\F{\mathfrak{F}}
\def\G{\mathfrak{G}}

\def\Ch{\mathcal{C}}
\def\Dh{\mathcal{D}}

\newcommand {\Span}{\operatorname{Span}}

\newcommand{\Sym}{\operatorname{Sym}}

\begin{document}

\title{Levi components of parabolic subalgebras of finitary Lie algebras}

\author{Elizabeth Dan-Cohen}
\address{Jacobs University Bremen, Campus Ring 1, 28759 Bremen, Germany}
\email{elizabeth.dancohen@gmail.com}

\author{Ivan Penkov}
\address{Jacobs University Bremen, Campus Ring 1, 28759 Bremen, Germany}
\email{i.penkov@jacobs-university.de}

\thanks{Both authors acknowledge the partial support of the DFG through Grants PE 980/2-1 and PE 980/3-1.}


\date{August 1, 2010}


\begin{abstract}
We characterize locally semisimple subalgebras $\l$ of $\sl_\infty$, $\so_\infty$, and $\sp_\infty$ which are Levi components of parabolic subalgebras.  Given $\l$, we describe the parabolic subalgebras $\p$ such that $\l$ is a Levi component of $\p$.  We also prove that not every maximal locally semisimple subalgebra of a finitary Lie algebra is a Levi component.

When the set of self-normalizing parabolic subalgebras $\p$ with fixed Levi component $\l$ is finite, we prove an estimate on its cardinality.  We consider various examples which highlight the differences from the case of parabolic subalgebras of finite-dimensional simple Lie algebras.
\end{abstract}

\maketitle

\section{Introduction}

The foundations of the theory of finitary Lie algebras have been laid in \cite{Baranov1, Baranov2, BaranovStrade, PStrade}.  This has made possible the development of a more detailed structure theory for the finitary Lie algebras \cite{NeebP, DPSnyder, DimitrovP1, D, DP, DPWolf, DimitrovP2}.  In particular, the notions of Levi components and parabolic subalgebras were developed for finitary Lie algebras in \cite{DP}.  Nevertheless, the problem of an explicit description of all Levi component of parabolic subalgebras was not addressed there.  This is the purpose of the present paper.  More precisely, we identify the subalgebras which occur as the Levi component of a simple finitary Lie algebra, and we characterize all parabolic subalgebras of which a given subalgebra is a Levi component.  In addition,  we provide criteria for the number of self-normalizing parabolic subalgebra with a prescribed Levi component to be finite; note that the finite numbers which occur can be quite unlike those in the finite-dimensional case. 

Along the way we present examples to highlight the many differences between the finitary and finite-dimensional situations.  One phenomenon seen here for the first time is a maximal locally semisimple subalgebra of a parabolic subalgebra which is not a Levi component of the parabolic subalgebra.  This answers a question left open in \cite{DP}.  It follows immediately that a maximal locally reductive subalgebra of a parabolic subalgebra is not in general a locally reductive part of the parabolic subalgebra.  Moreover, we give examples in which a maximal locally reductive subalgebra of a parabolic subalgebra, despite containing a Levi component, nevertheless is not a locally reductive part of the parabolic subalgebra.

\section{Preliminaries}

\subsection{Background on locally finite Lie algebras}

Let $V$ and $V_*$ be countable-dimensional vector spaces over the complex numbers, together with a nondegenerate pairing $\langle \cdot , \cdot \rangle : V \times V_* \rightarrow \C$.  A subspace $F \subset V$ is said to be \emph{closed} (in the Mackey topology) if $F = F^{\perp \perp}$.  By a result of Mackey \cite{Mackey}, the vector spaces $V$ and $V_*$ admit dual bases: that is, there are bases $\{v_i \mid i \in I\}$ and $\{ v^i \mid i \in I \}$ of $V$ and $V_*$, respectively, such that $\langle v_i , v^j \rangle = \delta_{ij}$.  We denote by $\gl(V,V_*)$ the Lie algebra associated to the associative algebra $V \otimes V_*$ with multiplication
\begin{align*}
V \otimes V_* \times V \otimes V_* & \rightarrow V \otimes V_* \\
(v \otimes w , v' \otimes w' ) & \mapsto \langle v' , w \rangle v \otimes w'.
\end{align*}
The vectors $v_i \otimes v^j$ form a basis for $\gl(V,V_*)$, with commutation relations $[ v_i \otimes v^j , v_k \otimes v^l ] = \delta_{jk} v_i \otimes v^l - \delta_{il} v_k \otimes v^j$.  Thus $\gl(V,V_*) \cong \gl_\infty$, where $\gl_\infty$ is the direct limit of the system
\begin{eqnarray*}
\gl_n \rightarrow \gl_{n+1} && A \mapsto \left( \begin{matrix} A & 0 \\ 0 & 0 \end{matrix} \right) .
\end{eqnarray*}
The Lie algebras $\sl_\infty$, $\so_\infty$, and $\sp_\infty$ are similarly defined; that is, they are direct limits of systems of finite-dimensional simple Lie algebras where the natural representation of each successive Lie algebra considered as a representation of the previous Lie algebra decomposes as a direct sum of the natural representation plus a trivial representation.  We denote by $\sl(V,V_*)$ the commutator subalgebra of $\gl(V,V_*)$, so $\sl(V,V_*) \cong \sl_\infty$.

Suppose $V = V_*$.  When the pairing $V \times V \rightarrow \C$ is symmetric, then we denote by $\so(V)$ the Lie algebra $\Lambda^2 V \cong \so_\infty$.  When the pairing $V \times V \rightarrow \C$ is antisymmetric, then we denote by $\sp(V)$ the Lie algebra $\Sym^2 V \cong \sp_\infty$.   A subspace $F \subset V$ is called \emph{isotropic} if $\langle F, F \rangle = 0$ and \emph{coisotropic} if $F^\perp \subset F$.

We fix notation for maps
\begin{align*}
\Lambda : \gl(V,V) & \rightarrow \so(V) \\
v \otimes w & \mapsto v \otimes w - w \otimes v \\
\intertext{and}
S : \gl(V,V) & \rightarrow \sp(V) \\
v \otimes w & \mapsto v \otimes w + w \otimes v,
\end{align*}
and note that they give homomorphisms of Lie algebras when restricted to $\gl(X,Y)$ for any isotropic subspaces $X$, $Y \subset V$ such that the restriction $\langle \cdot , \cdot \rangle |_{X \times Y}$ is nondegenerate.

We call $\s$ a \emph{standard special linear subalgebra} of $\gl(V,V_*)$ or $\sl(V,V_*)$ if 
$$\s = \sl(X,Y)$$
for some subspaces $X \subset V$ and $Y \subset V_*$ such that the restriction $\langle \cdot , \cdot \rangle |_{X \times Y}$
 is nondegenerate.  We call $\s$ a \emph{standard special linear subalgebra} of $\so(V)$ (resp., of $\sp(V)$) if 
$$\s = \Lambda( \sl(X,Y))$$
(resp., if $\s = S( \sl(X,Y) ) $)
for some isotropic subspaces $X$, $Y \subset V$ such that the restriction
$\langle \cdot , \cdot \rangle |_{X \times Y}$ is nondegenerate. 

A Lie algebra $\g$ is said to be \emph{finitary} if there exists a faithful countable-dimensional representation $\g \hookrightarrow \End W$ where $W$ has a basis in which the matrix of each endomorphism in the image of $\g$ has only finitely many nonzero entries.  Suppose $\g$ is a finitary Lie algebra, and let $\{ w_i \mid i \in I \}$ be a basis of a representation $W$ of $\g$ as in the definition of finitary.  For each i, define $w^i \in W^*$ by $w^i (w_j) := \delta_{ij}$;  let $W_*:= \Span \{ w^i \mid i \in I \}$, so that $W_*$ is a countable-dimensional subspace of the full algebraic dual space $W^*$.  Then the map $\g \hookrightarrow \End W$ factors through $\gl(W,W_*)$, yielding an injective homomorphism $\g \hookrightarrow \gl(W,W_*)$.

A Lie algebra $\g$ is \emph{locally finite} if every finite subset of $\g$ is contained in a finite-dimensional subalgebra.  A countable-dimensional locally finite Lie algebra $\g$ is therefore the direct limit of a system of injective homomorphisms of finite-dimensional Lie algebras $\g_n \hookrightarrow \g_{n+1}$ for $n \in \Z_{>0}$.  Observe that any finitary Lie algebra, being isomorphic to a subalgebra of $\gl_\infty$, is itself locally finite.

A locally finite Lie algebra $\g$ is \emph{locally solvable} (respectively, \emph{locally nilpotent}) if every finite-dimensional subalgebra of $\g$ is solvable (resp., nilpotent).  A locally finite Lie algebra $\g$ is \emph{locally simple} (resp., \emph{locally semisimple}) if there exists a system of finite-dimensional simple (resp., semisimple) subalgebras $\g_i$ of which $\g$ is the direct limit.  A locally finite Lie algebra $\g$ is \emph{locally reductive} if it is the direct limit of some system of finite-dimensional reductive subalgebras $\g_n \hookrightarrow \g_m$ where $\g_m$ is a semisimple $\g_n$-module for all $n < m$. 

A subalgebra $\g$ of a locally finite Lie algebra is called \emph{parabolic} if it contains a maximal locally solvable subalgebra of $\g$.

Let $\q$ be a subalgebra of a locally reductive Lie algebra $\g$, and let $\r$ denote the locally solvable radical of $\q$.  The \emph{linear nilradical} $\n_\q$ of $\q$ is the set of Jordan nilpotent elements of $\r$ (see \cite{DPSnyder} for the details of Jordan decomposition in a locally reductive Lie algebra).  One may check that $\n_\q$ is a locally nilpotent ideal of $\q$; see \cite{DP} where the proof is given in the finitary case, and it generalizes.  

A subalgebra $\l \subset \q$ is a \emph{Levi component} of $\q$ if $[\q,\q] = (\r \cap [\q,\q]) \subsetplus \l$.  A locally reductive subalgebra $\q_{red}$ of $\q$ is a \emph{locally reductive part} of $\q$ if $\q = \n_\q \subsetplus \q_{red}$.

A subalgebra $\q \subset \g$ is \emph{splittable} if the nilpotent and semisimple parts of each element of $\q$ are also in $\q$.  The fact that every splittable subalgebra of $\gl_\infty$ has a locally reductive part was shown in \cite{DP}.  If $\q \subset \gl_\infty$ is splittable, then any subalgebra of $\q$ containing $\n_\q + [\q,\q]$ is said to be \emph{defined by trace conditions} on $\q$ \cite{DP}.  (Note that any vector space containing $\n_\q + [\q,\q]$ and contained in $\q$ is a subalgebra of $\g$.)

Let $X$ be a vector space, and let $\Ch$ be a set of subspaces of $X$ on which inclusion gives a total order.  Suppose $F' \subsetneq F''$ are subspaces of $X$ in $\Ch$.  We call $F'$ the \emph{immediate predecessor} of $F''$ in $\Ch$ if for all $C \in \Ch$ either $C \subset F'$ or $F'' \subset C$.  When $F'$ is the immediate predecessor of $F''$ in $\Ch$, we also say that $F''$ is the \emph{immediate successor} of $F'$ in $\Ch$, and that $F' \subset F''$ are an \emph{immediate predecessor-successor pair} in $\Ch$.  

\begin{definition} \cite{DimitrovP1}
A set $\F$ of subspaces of $X$ for which inclusion gives a total order is called a \emph{generalized flag} if the following two conditions hold:
\begin{enumerate}
\item For all $F \in \F$, there is an immediate predecessor-successor pair $F' \subset F''$ in $\F$ such that $F \in \{ F' , F'' \}$;
\item For all nonzero $v \in X$, there is an immediate predecessor-successor pair $F' \subset F''$ in $\F$ such that $v \in F'' \setminus F'$.
\end{enumerate}
\end{definition}

For any generalized flag $\F$, we denote by $A$ the set of immediate predecessor-successor pairs of $\F$.  Then by definition we have $\F = \{F'_\alpha , F''_\alpha \}_{\alpha \in A}$, where $F'_\alpha$ is the immediate predecessor of $F''_\alpha$ in the inclusion order, and the two subspaces are the pair $\alpha \in A$.   Similarly, we denote by $B$ the set of immediate predecessor-successor pairs of any generalized flag denoted by $\G$, so that $\G = \{G'_\beta, G''_\beta \}_{\beta \in B}$.  

For any generalized flag $\F$ in $V$, the stabilizer of $\F$ in $\gl(V,V_*)$ is denoted by $\St_\F$ and is given by the formula $\St_\F = \sum_{\alpha \in A} F''_\alpha \otimes (F'_\alpha)^\perp$ \cite{DimitrovP1}.

\subsection{Background on parabolic subalgebras of finitary Lie algebras}

Recall that there is a nondegenerate pairing $\langle \cdot , \cdot \rangle : V \times V_* \rightarrow \C$.
We say that a generalized flag $\F$ in $V$ is \emph{semiclosed} if $(F')^{\perp \perp} \in \{ F' , F'' \}$
for each immediate predecessor-successor pair $F' \subset F''$ in $\F$.  

\begin{prop} \label{genf}
 Let $\Ch$ be a set of subspaces of $V$ totally ordered by inclusion.  Then the following exist:
\begin{enumerate}
\item a generalized flag in $V$ with the same $\gl(V,V_*)$-stabilizer as $\Ch$ \cite{DimitrovP1};
\item \label{uniquesemic}  a unique semiclosed generalized flag in $V$ with the same $\gl(V,V_*)$-stabilizer as $\Ch$, if each nonclosed subspace in $\Ch$ is the immediate successor in $\Ch$ of a closed subspace.
\end{enumerate}
\end{prop}

\begin{proof}
We recall the construction from \cite{DimitrovP1} that produces a generalized flag with the same stabilizer as a given set $\Ch$ of subspaces totally ordered by inclusion.  For a fixed nonzero vector $x$ in $V$, consider the subspace $F'(x)$ which is the union of the subspaces in $\Ch$ not containing $x$; it is properly contained in $F''(x)$, the intersection of the subspaces in $\Ch$ containing $x$.  The set of subspaces of the form $F'(x)$ or $F''(x)$ as $x$ runs over the nonzero vectors in $V$ is a generalized flag with the same stabilizer as $\Ch$.

Assume that each nonclosed subspace in $\Ch$ has an immediate predecessor in $\Ch$, and the latter is closed.  Then the set $\Dh := \Ch \cup \{ X^{\perp \perp} \mid X \in \Ch \}$ is totally ordered by inclusion.  To prove this, it suffices to show that $X_1 \subsetneq X_2$ implies $X_1^{\perp \perp} \subset X_2$ for all $X_1$, $X_2 \in \Ch$.  If $X_2$ is closed, then this is clear.  If $X_2$ is not closed, then by assumption $X_2$ has an immediate predecessor $X_3 \in \Ch$, and $X_3$ is closed;  then $X_1 \subset X_3 \subset X_2$, so $X_1^{\perp \perp} \subset X_3 \subset X_2$.

For each nonzero vector $x \in V$, we define 
\begin{align*}
F_1(x) & :=  \bigcup_{Y \in \Dh, \, x \notin Y} Y &
F_2 (x) &:= F_1(x)^{\perp \perp} &
F_3 (x) & :=  \bigcap_{Y \in \Dh, \, x \in Y} Y .
\end{align*}
Applying the general construction from \cite{DimitrovP1} to $\Dh$ yields the generalized flag
$ \left\{ F_1(x), \, F_3(x)  \mid 0 \neq x \in V \right\}$.
We claim that the refinement $$\F := \left\{ F_1(x), \, F_2(x) , \, F_3(x)  \mid 0 \neq x \in V \right\}$$
is a semiclosed generalized flag with the same stabilizer as $\Ch$.

To see that $\F$ is a generalized flag, it suffices to show that $F_2 (x) \subset F_3(x)$ for all nonzero $x \in V$.  If $F_3 (x)$ is closed, then this is clear.  If $F_3 (x)$ is not closed, then $F_3(x) \in \Ch$ (otherwise it would be the intersection of the closed subspaces containing it, since each nonclosed subspace in $\Ch$ is assumed to have an immediate predecessor in $\Ch$ which is closed); hence $F_3(x)$ is the immediate successor in $\Ch$ of a closed subspace, and the latter is then $F_1 (x) = F_2 (x)$.

Now each immediate predecessor-successor pair in $\F$ has the form $F_1 (x) \subset F_2(x)$ or $F_2(x) \subset F_3(x)$ for some nonzero $x \in V$.  In either case the condition defining a semiclosed generalized flag is satisfied.  By construction $\Ch$ and $\F$ have the same stabilizer.  Finally, the uniqueness of $\F$ follows from \cite[Proposition 3.8]{DP}
\end{proof}

We say that semiclosed generalized flags $\F$ in $V$ and $\G$ in $V_*$ form a \emph{taut couple} if $F^\perp$ is stable under $\St_\G$ for all $F \in \F$ and $G^\perp$ is stable under $\St_\F$ for all $G \in \G$.  If $V = V_*$, then a semiclosed generalized flag $\F$ in $V$ is called \emph{self-taut} if $F^\perp$ is stable under $\St_\F$ for all $F \in \F$.  In the interest of clarity, we should emphasize that $\St_\F$ means the $\gl(V,V)$-stabilizer of $\F$ in the case $V=V_*$.

We now summarize Theorem 5.6 in \cite{DP}.  For any taut couple $\F$, $\G$, the subalgebra $\St_\F \cap \St_\G$ is a self-normalizing parabolic subalgebra of $\gl(V,V_*)$.  Moreover, the self-normalizing parabolic subalgebras of $\gl(V,V_*)$ are in bijection with the taut couples in $V$ and $V_*$ \cite[Corollary 5.7]{DP}.  If $\p$ is any parabolic subalgebra of $\gl(V,V_*)$, then the normalizer of $\p$ is a self-normalizing parabolic subalgebra, which we denote $\p_+$; furthermore, $\p$ is defined by trace conditions on $\p_+$.  We call the (unique) taut couple $\F$, $\G$ such that $\p_+ = \St_\F \cap \St_\G$ the \emph{taut couple associated to $\p$}.  The smallest parabolic subalgebra with the associated taut couple $\F$, $\G$ is denoted by $\p_-$, and it is the set of elements of $\St_\F \cap \St_\G$ such that each component in each infinite-dimensional block of a locally reductive part is traceless.

The situation for $\sl(V,V_*)$, $\so(V)$, and $\sp(V)$ is quite similar to the above, in that every parabolic subalgebra is defined by trace conditions on its normalizer, which is a self-normalizing parabolic subalgebra.  Theorem 5.6 in \cite{DP} also characterizes the parabolic subalgebras of $\sl(V,V_*)$, as follows.  The self-normalizing parabolic subalgebras of $\sl(V,V_*)$ are also in bijection with the taut couples in $V$ and $V_*$, where the joint stabilizer $\St_\F \cap \St_\G \cap \sl(V,V_*)$ is the self-normalizing parabolic subalgebra of $\sl(V,V_*)$ corresponding to the taut couple $\F$, $\G$.  The parabolic subalgebras of $\so(V)$ and $\sp(V)$ are described in Theorem 6.6 in \cite{DP}. In the case of $\sp(V)$, taking the stabilizer gives yet again a bijection between the self-taut generalized flags in $V$ and the self-normalizing parabolic subalgebras of $\sp(V)$.  In the case of $\so(V)$, the analogous map surjects onto the self-normalizing parabolic subalgebras of $\so(V)$, but by contrast it is not injective.  The fibers of size different from $1$ are all of size $3$ \cite{DPWolf}.  Note that the claim in \cite[Theorem 6.6]{DP} regarding the uniqueness in the $\so_\infty$ case is erroneous; the correct statement is \cite[Theorem 2.8]{DPWolf}.

\subsection{Locally reductive parts of parabolic subalgebras of $\sl_\infty$ and $\gl_\infty$}

We denote by $C$ the ordered subset  $$C:= \{ \alpha \in A \mid (F'_\alpha)^{\perp \perp} = F'_\alpha \}$$ for any semiclosed generalized flag $\F$ in $V$ with $V \neq V_*$.  For any taut couple $\F$, $\G$ there is a natural bijection between $C$ and the set $\{ \beta \in B \mid (G'_\beta)^{\perp \perp} = G'_\beta \}$, under which $F'_\gamma = (G''_\gamma)^\perp$ for all $\gamma \in C$ \cite[Proposition 3.4]{DP}.  This enables us to consider $C$ as a subset of $B$, as well; note that the inclusion of $C$ into $B$ is order reversing.

In the case $V = V_*$ we denote by $C$ the analogous subset $$C := \{ \alpha \in A \mid (F'_\alpha)^{\perp \perp} = F'_\alpha \textrm{, } F''_\alpha \subset (F''_\alpha)^\perp \}.$$  In this case, there is a natural order-reversing bijection between $C$ and the set $\{ \alpha \in A \mid (F'_\alpha)^{\perp \perp} = F'_\alpha \supset (F'_\alpha)^\perp \}$ \cite[Proposition 6.1]{DP}.  For $\gamma \in C$, we denote by $G'_\gamma \subset G''_\gamma$ the corresponding pair where $G'_\gamma$ is closed and coisotropic, and thus obtain the analogous statements $F'_\gamma = (G''_\gamma)^\perp$ and $G'_\gamma = (F''_\gamma)^\perp$ for all $\gamma \in C$.

The next theorem is slightly more general than Proposition 3.6 (ii) in \cite{DP}.  Note that subspaces $X_\gamma$ and $Y_\gamma$ satisfying the hypotheses necessarily exist \cite[Prop. 3.6 (ii)]{DP}.

\begin{theorem} \label{strongerhypotheses}
Let $\p$ be a parabolic subalgebra of $\sl(V,V_*)$ or $\gl(V,V_*)$, with the associated taut couple $\F$, $\G$. 
Let $X_\gamma \subset V$ and $Y_\gamma \subset V_*$ be any subspaces with $$F''_\gamma = F'_\gamma \oplus X_\gamma \textrm{ and } G''_\gamma = G'_\gamma \oplus Y_\gamma$$ for all $\gamma \in C$, such that
$\langle X_\gamma , Y_\eta \rangle = 0 \textrm{ for } \gamma \neq \eta.$
Then $\p \cap \bigoplus_{\gamma \in C} \gl(X_\gamma, Y_\gamma)$ is a locally reductive part of $\p$.
\end{theorem}

\begin{proof}
The proof of Proposition 3.6 (ii) in \cite{DP} shows that $\bigoplus_{\gamma \in C} \gl(X_\gamma, Y_\gamma)$ is a locally reductive part of $\St_\F \cap \St_\G$.  Note that there are additional hypotheses in \cite{DP} which were not used in the proof.  Intersecting a locally reductive part of $\St_\F \cap \St_\G$ with $\p$ imposes the same trace conditions defining $\p$, yielding a locally reductive part of $\p$.
\end{proof}

\section{Parabolic subalgebras of $\sl_\infty$ and $\gl_\infty$ with given Levi component}

In this section we prove the two main theorems of this paper.  Theorem~\ref{firstcharacterization} identifies the subalgebras of $\sl(V,V_*)$ and $\gl(V,V_*)$ which can be realized as the Levi component of a parabolic subalgebra.  Theorem~\ref{iffgl} characterizes all parabolic subalgebras of which a given subalgebra $\l$ is a Levi component.

Every parabolic subalgebra of $\sl(V,V_*)$ or $\gl(V,V_*)$ has a Levi component of the form $\bigoplus_{i \in I} \sl_{n_i}$ for some $n_i \in \Z_{\geq 2} \cup \{ \infty \}$, by Theorem~\ref{strongerhypotheses}.  We therefore consider whether every such subalgebra of $\sl(V,V_*)$ or $\gl(V,V_*)$ is a Levi component of some parabolic subalgebra.  An obstruction presents itself immediately, in consequence of the following lemma.

\begin{lemma} \label{notdiagonal}
Let $\p$ be a parabolic subalgebra of $\gl(V,V_*)$, and $\l$ a Levi component of $\p$.  Then $\l$ is a direct sum of standard special linear subalgebras.  

Furthermore, the order of the generalized flag $\F$ in $V$ associated to $\p$ induces an order on the simple direct summands of $\l$.
\end{lemma}

\begin{proof}
 The Levi component $\l$ of $\p$ is a maximal locally semisimple subalgebra of $\p$ by \cite[Theorem 4.3]{DP}.
Since $\l$ is a locally semisimple subalgebra of $\g$, it is a direct sum of simple subalgebras \cite{DimitrovP2}.  Let $\s$ denote one of the simple direct summands of $\l$, and take $\l_0$ to be the direct sum of all the other simple direct summands of $\l$, so $\l = \s \oplus \l_0$.
We will show that $\s$ is a standard special linear subalgebra.

When $\s$ is finite dimensional, there exist nontrivial simple $\s$-submodules $X_1$, $X_2$, \ldots, $X_k$ of $V$ and $Y_1$, $Y_2$, \ldots, $Y_k$ of $V_*$ such that $\langle X_i , Y_j \rangle = 0$ for $i \neq j$ and $$\s \subset \sl(X_1 , Y_1) \oplus \sl(X_2, Y_2) \oplus \cdots \oplus \sl(X_k, Y_k).$$  When $\s$ is infinite dimensional, the same statement follows from \cite{DimitrovP2}, where the authors characterize arbitrary subalgebras of $\gl(V,V_*)$ isomorphic to $\sl_\infty$, $\so_\infty$, and $\sp_\infty$.

We will show that $\sl(X_1 , Y_1) \oplus \sl(X_2, Y_2) \oplus \cdots \oplus \sl(X_k, Y_k) \subset \p$.  As the labeling is arbitrary, it is enough to show that $\sl(X_1, Y_1) \subset \p$.  Moreover, it suffices to show that $\F$ is stable under $\sl(X_1, Y_1)$, where $\F$, $\G$ is the taut couple associated to $\p$.  Indeed, it then follows by symmetry that $\G$ is also stable under $\sl(X_1, Y_1)$, and hence $\sl(X_1, Y_1) \subset [ \St_\F \cap \St_\G , \St_\F \cap \St_\G ] \subset \p$.

Fix a nonzero vector $x_i \in X_i$ for $i = 1, \ldots, k$.  By the definition of a generalized flag $x_i \in F''_{\alpha_i} \setminus F'_{\alpha_i}$ for some $\alpha_i \in A$.  Consider that $X_i =\s \cdot x_i \subset \St_\F \cdot x_i \subset F''_{\alpha_i}$.
If there were a nonzero vector in the intersection $X_i \cap F'_{\alpha_i}$, then one would have similarly that $X_i \subset F''_{\beta_i}$ for some $\beta_i < \alpha_i$, contradicting the fact that $x_i \notin F'_{\alpha_i}$.  Thus we conclude $X_i \cap F'_{\alpha_i} = 0$, and we have shown $F'_{\alpha_i} \oplus X_i \subset F''_{\alpha_i}$.  

Observe that $\s \cdot F'_{\alpha_1} \subset \St_\F \cdot F'_{\alpha_1} \subset F'_{\alpha_1}$.
Let $\pi_i$ denote the $i$th projection for the decomposition $\sl(X_1 , Y_1) \oplus \sl(X_2, Y_2) \oplus \cdots \oplus \sl(X_k, Y_k)$.  One has $\pi_i(\s) \cdot F'_{\alpha_1} \subset ( X_i \otimes Y_i ) \cdot V \subset X_i$ for each $i$, and hence $\pi_1 (\s) \cdot F'_{\alpha_1} = 0$.  Since $(Y_1)^\perp$ is the largest trivial $\pi_1(\s)$-submodule of $V$, we see that $F'_{\alpha_1} \subset (Y_1)^\perp$.  As a result
\begin{equation*}\sl(X_1 , Y_1) \subset X_1 \otimes Y_1 \subset F''_{\alpha_1} \otimes (F'_{\alpha_1})^\perp \subset \St_\F. 
\end{equation*}

Therefore $\sl(X_1 , Y_1) \oplus \sl(X_2, Y_2) \oplus \cdots \oplus \sl(X_k, Y_k) \oplus \l_0$ is a locally semisimple subalgebra containing $\l$ and contained in $\p$.  By the maximality of $\l$, we obtain $\s = \sl(X_1 , Y_1) \oplus \sl(X_2, Y_2) \oplus \cdots \oplus \sl(X_k, Y_k)$ and in particular $k = 1$, so $\s$ is a standard special linear subalgebra.

This shows that $\l = \bigoplus_{i \in I} \sl (X_i , Y_i)$ for some subspaces $X_i \subset V$ and $Y_i \subset V_*$.  Fix $i \in I$.  As shown above, there exists $\alpha_i \in A$ such that $F'_{\alpha_i} \oplus X_i \subset F''_{\alpha_i}$ and $\langle F'_{\alpha_i} , Y_i \rangle = 0$.  Similarly, there exists $\beta_i \in B$ such that $G'_{\beta_i} \oplus Y_i \subset G''_{\beta_i}$ and $\langle X_i , G'_{\beta_i} \rangle = 0$.  Since $\F$, $\G$ form a taut couple, it follows that $F'_{\alpha_i} = (G''_{\beta_i})^\perp$ and $G'_{\beta_i} = (F''_{\alpha_i})^\perp$, and hence $\alpha_i \in D$.  Thus the rule $i \mapsto \alpha_i$ gives a well-defined map $\kappa : I \rightarrow D$.  We claim that $\kappa$ is an injective map.  To see this, suppose $\kappa (j) = \kappa (k)$ for some $j, k \in I$.  Then $\l \subset \sl (X_j + X_k  , Y_j + Y_k) \oplus \bigoplus_{i \neq j,k} \s_i \subset \p$, so the maximality of $\l$ implies that $j = k$.  Since $D$ is an ordered set, there is an induced order on $I$, the set of direct summands of $\l$.
\end{proof}

Theorem~\ref{firstcharacterization} below shows that there are no further obstructions to finding a parabolic subalgebra such that a given subalgebra is a Levi component.  We first prove a lemma and a proposition.

\begin{lemma} \label{tripleperp}
Fix subspaces $X \subset V$ and $Y \subset V_*$ such that the restriction $\langle \cdot , \cdot \rangle |_{X \times Y}$ is nondegenerate.  Let $T \subset V$, and define $U := ((T + X)^\perp \oplus Y)^\perp$.  Then $$U = ((U \oplus X)^\perp \oplus Y)^\perp.$$
\end{lemma}

\begin{proof}
To see that $U \subset ((U \oplus X)^\perp \oplus Y)^\perp = (U \oplus X)^{\perp \perp} \cap Y^\perp$, consider that $U = ((T + X)^\perp \oplus Y)^\perp = (T + X)^{\perp \perp} \cap Y^\perp \subset Y^\perp$, while the inclusion $U \subset (U + X)^{\perp \perp}$ is automatic.  For the reverse containment, we observe that 
\begin{align*}
\langle ((T + X)^\perp \oplus Y)^\perp \oplus X , (T + X)^\perp \rangle &  =  \langle (T + X)^{\perp \perp} \cap Y^\perp \oplus X , (T + X)^\perp \rangle \\
& \subset \langle (T + X)^{\perp \perp}, (T + X)^\perp \rangle = 0.
\end{align*}
This shows that $(T + X)^\perp \subset (((T + X)^\perp \oplus Y)^\perp \oplus X)^\perp$.  Hence 
$$((((T + X)^\perp \oplus Y)^\perp \oplus X)^\perp \oplus Y)^\perp \subset ((T + X)^\perp \oplus Y)^\perp,$$
i.e.\ $((U \oplus X)^\perp \oplus Y)^\perp \subset U$.
\end{proof}

For any semiclosed generalized flag $\F$ we set 
$$D := \{ \gamma \in C \mid \dim F''_\gamma / F'_\gamma > 1 \}.$$  
Note that $D = \{ \gamma \in C \mid \dim G''_\gamma / G'_\gamma > 1 \}$, as the pairing $V \times V_* \rightarrow \C$ induces a nondegenerate pairing of $F''_\gamma / F'_\gamma$ and $G''_\gamma / G'_\gamma$ for all $\gamma \in C$.

\begin{prop} \label{weakerhypotheses}
Let $\p$ be a parabolic subalgebra of $\sl(V,V_*)$ or $\gl(V,V_*)$, with the associated taut couple $\F$, $\G$. 
Let $X_\gamma \subset V$ and $Y_\gamma \subset V_*$ be any subspaces with $$F''_\gamma = F'_\gamma \oplus X_\gamma \textrm{ and } G''_\gamma = G'_\gamma \oplus Y_\gamma$$ for all $\gamma \in D$, such that
$\langle X_\gamma , Y_\eta \rangle = 0 \textrm{ for } \gamma \neq \eta.$
Then $\bigoplus_{\gamma \in D} \sl(X_\gamma, Y_\gamma)$ is a Levi component of $\p$.
\end{prop}

\begin{proof}
Let $X_\gamma$ and $Y_\gamma$ for $\gamma \in D$ be as in the statement, and let $\tilde{X}_\gamma$ and $\tilde{Y}_\gamma$ for $\gamma \in C$ be as in Theorem~\ref{strongerhypotheses}, so that $\St_\F \cap \St_\G = \n_\p \subsetplus \bigoplus_{\gamma \in C} \gl (\tilde{X}_\gamma , \tilde{Y}_\gamma)$.  Clearly the subalgebra $\n_\p \subsetplus \bigoplus_{\gamma \in D}  \gl (\tilde{X}_\gamma , \tilde{Y}_\gamma)$ is defined by trace conditions on $\St_\F \cap \St_\G$.  
A subalgebra has the same set of Levi components as any subalgebra defined by trace conditions on it, by \cite[Proposition 4.9]{DP}.  Since $\p$ is also defined by trace conditions on $\St_\F \cap \St_\G$, all three have the same set of Levi components, and it suffices to show that $\bigoplus_{\gamma \in D} \sl(X_\gamma, Y_\gamma)$ is a Levi component of  $\n_\p \subsetplus \bigoplus_{\gamma \in D}  \gl (\tilde{X}_\gamma , \tilde{Y}_\gamma)$.

Clearly $\bigoplus_{\gamma \in D} \sl(X_\gamma, Y_\gamma)$ is a Levi component of  $\n_\p \subsetplus \bigoplus_{\gamma \in D}  \gl (X_\gamma, Y_\gamma)$.  We claim that $$\n_\p \subsetplus \bigoplus_{\gamma \in D}  \gl (\tilde{X}_\gamma , \tilde{Y}_\gamma) = \n_\p \subsetplus \bigoplus_{\gamma \in D}  \gl (X_\gamma , Y_\gamma).$$
To see this, consider that for each $\gamma \in C$,
\begin{align*}
\tilde{X}_\gamma \otimes \tilde{Y}_\gamma \subset F''_\gamma \otimes G''_\gamma
& = (F'_\gamma \oplus X_\gamma) \otimes (G'_\gamma \oplus Y_\gamma) \\
& = (F'_\gamma \otimes G''_\gamma + F''_\gamma \otimes G'_\gamma) \oplus X_\gamma \otimes Y_\gamma.
\end{align*}
Since $F'_\gamma \otimes G''_\gamma + F''_\gamma \otimes G'_\gamma \subset \n_\p$, we have shown that
$\gl (\tilde{X}_\gamma , \tilde{Y}_\gamma) \subset \n_\p \subsetplus \gl (X_\gamma , Y_\gamma)$.  One has symmetrically that $\gl (X_\gamma , Y_\gamma) \subset \n_\p \subsetplus \gl (\tilde{X}_\gamma , \tilde{Y}_\gamma)$ for each $\gamma \in C$.
\end{proof}

Let $\g$ denote either $\sl(V,V_*)$ or $\gl(V,V_*)$.

\begin{theorem} \label{firstcharacterization}
Let $\l$ be a subalgebra of $\g$.  There exists a parabolic subalgebra $\p$ of $\g$ such that $\l$ is a Levi component of $\p$ if and only if $\l$ is a direct sum of standard special linear subalgebras of $\g$.  

Moreover, given a subalgebra $\l$ for which such parabolic subalgebras exist, one exists that induces an arbitrary order on the simple direct summands of $\l$ (see Lemma~\ref{notdiagonal}).
\end{theorem}

\begin{proof}
In this and subsequent proofs, we assume (without loss of generality) that $\g = \gl(V,V_*)$.  The only if direction was proved in Lemma~\ref{notdiagonal}.  

Conversely, fix commuting standard special linear subalgebras $\s_i \subset \gl(V,V_*)$ for $i \in I$, as well as an order on $I$.  We will construct a parabolic subalgebra $\p$ such that $\l := \bigoplus_{i \in I} \s_i$ is a Levi component of $\p$, and $\p$ induces the given order on $I$.  Each standard special linear subalgebra $\s_i$ determines subspaces $X_i \subset V$ and $Y_i \subset V_*$ such that $\s_i = \sl(X_i , Y_i)$.  As these direct summands commute, it must be that $\langle X_i , Y_j \rangle = 0$ for $i \neq j$.

For each $i$, we define $$U_i := (( \bigoplus_{k \leq i} X_k)^\perp \oplus Y_i)^\perp.$$
One may check in an elementary fashion that $U_i \oplus X_i \subset U_j$ for all $i < j$.  Since $U_j$ is closed for all $j \in I$, we have moreover that $U_i \oplus X_i \subset (U_i \oplus X_i)^{\perp \perp} \subset U_j$ for all $i < j$.  Furthermore, for each $i \in I$, an application of Lemma~\ref{tripleperp} using $T = \bigoplus_{k<i} X_k$ shows that
$U_i = ((U_i \oplus X_i)^\perp \oplus Y_i)^\perp.$

We claim that there is a unique semiclosed generalized flag $\F_0$ in $V$ with the same stabilizer as the set $\{ U_i , \, U_i \oplus X_i \mid i \in I \}$.  This follows from Proposition~\ref{genf} (\ref{uniquesemic}).
Similarly, there is a unique semiclosed generalized flag in $V_*$ with the same stabilizer as the set
$$\{ (U_i \oplus X_i)^\perp , \, (U_i \oplus X_i)^\perp \oplus Y_i \mid i \in I \}.$$
One may check that $\F_0$, $\G_0$ form a taut couple using the identity $U_i = ((U_i \oplus X_i)^\perp \oplus Y_i)^\perp$.
Indeed, $\F_0$, $\G_0$ is the minimal taut couple with immediate predecessor-successor pairs $U_i \subset U_i \oplus X_i$ in $\F$ and $(U_i \oplus X_i)^\perp \subset (U_i \oplus X_i)^\perp \oplus Y_i$ in $\G$ for all $i \in I$.

Let $\F$, $\G$ be maximal among the taut couples having immediate predecessor-successor pairs $U_i \subset U_i \oplus X_i$ in $\F$ and $(U_i \oplus X_i)^\perp \subset (U_i \oplus X_i)^\perp \oplus Y_i$ in $\G$ for all $i \in I$.  Then there is a natural bijection between $I$ and $D$.  By \cite[Theorem 5.6]{DP}, $\St_\F \cap \St_\G$ is a parabolic subalgebra of $\gl(V,V_*)$.  Moreover, it follows from Proposition~\ref{weakerhypotheses} that $\l$ is a Levi component of $\St_\F \cap \St_\G$.  By construction, the induced order on the simple direct summands of $\l$ is the given order on $I$.
\end{proof}

We now characterize the parabolic subalgebras of which a given subalgebra is a Levi component.

\begin{theorem} \label{iffgl}
Let $\p$ be a parabolic subalgebra of $\g$, with the associated taut couple $\F$, $\G$.  Then $\l$ is a Levi component of $\p$ if and only if
there exist subspaces $X_\gamma \subset V$ and $Y_\gamma \subset V_*$ with 
\begin{equation} \label{comps}
F''_\gamma = F'_\gamma \oplus X_\gamma  \textrm{ and }  G''_\gamma = G'_\gamma \oplus Y_\gamma
\end{equation}
for all $\gamma \in D$, such that 
$\l = \bigoplus_{\gamma \in D} \sl(X_\gamma, Y_\gamma)$.
\end{theorem}

\begin{proof}
One direction is Proposition~\ref{weakerhypotheses}.  Conversely, assume $\l$ is an arbitrary Levi component of $\p$.  By Theorem~\ref{firstcharacterization}, we have $\l = \bigoplus_{i \in I} \s_i$ for some standard special linear subalgebras $\s_i \subset \gl(V,V_*)$.  Hence there exist subspaces $X_i \subset V$ and $Y_i \subset V_*$ such that $\s_i = \sl(X_i, Y_i)$; again, $\langle X_i , Y_j \rangle = 0$ for $i \neq j$.
As shown in the proof of Lemma~\ref{notdiagonal}, there exists an injective map $\kappa : I \rightarrow D$, with the properties $F'_{\kappa(i)} \oplus X_i \subset F''_{\kappa(i)}$ and $\langle F'_{\kappa(i)} , Y_i \rangle = 0$. 

Let  $\tilde{X}_\gamma$ and $\tilde{Y}_\gamma$ for $\gamma \in D$ be as in the statement of Proposition~\ref{weakerhypotheses}.  Then since the span of the linear nilradical of $\p$ and any Levi component equals $\n_\p + [\p,\p]$ \cite{DP}, we have
$$\n_\p \subsetplus \bigoplus_{\gamma \in D} \sl(\tilde{X}_\gamma , \tilde{Y}_\gamma) = \n_\p \subsetplus \bigoplus_{i \in I} \sl(X_i, Y_i).$$

Let $d \in D$ be arbitrary.  Fix $v \in  \tilde{X}_d$.  Then
$$F'_d  + \big( \n_\p \subsetplus \bigoplus_{\gamma \in D} \sl(\tilde{X}_\gamma , \tilde{Y}_\gamma) \big) \cdot v  =  F'_d + \tilde{X}_d = F''_d,$$
since $\n_\p \cdot v \subset F'_d$.  So
\begin{align*}
F''_d  = F'_d + \big( \n_\p \subsetplus  \bigoplus_{i \in I} \sl(X_i, Y_i) \big) \cdot v 
& = F'_d + \bigoplus_{i \in I} \langle v , Y_i \rangle X_i \\
& = \begin{cases}
F'_d \oplus X_i & \textrm{if } d = \kappa (i) \textrm{ for some } i \in I\\
F'_d &  \textrm{if } \kappa (i) \neq d \textrm{ for all } i \in I,
\end{cases}
\end{align*}
since $\kappa (i) < d$ implies $X_i \subset F'_d$, while $\kappa (i) > d$ implies $\langle v , Y_i \rangle = 0$.  As $F'_d \subsetneq F''_d$, we conclude that $d = \kappa (i)$ for some $i \in I$.  Hence $\kappa$ is a bijection from $I$ to $D$.  Since we have shown that $F''_d = F'_d \oplus X_{\kappa^{-1} (d)}$ for all $d \in D$, we are done.
\end{proof}

Here is an example notably different from the finite-dimensional case.  Let $V$ and $V_*$ be vector spaces with bases $\{ v \} \cup \{ v_ i \mid i \in \Z_{> 0} \}$ and $\{ v_ i^* \mid i \in \Z_{> 0} \}$, pairing according to the rules
$\langle v_i , v_j^* \rangle = \delta_{ij}$ and  $$\langle v, v_j^* \rangle = 1 \textrm{ for all j.}$$
We will find all self-normalizing parabolic subalgebras of $\gl(V,V_*)$ with Levi component $\sl(X_1 , Y_1) \oplus \sl(X_2 , Y_2)$, where 
\begin{align*}
X_1 & :=  \Span \{ v_{2i-1} \mid i \in \Z_{>0} \} & Y_1 & :=  \Span \{ v_{2i-1}^* \mid i \in \Z_{>0} \} \\
 X_2 & :=   \Span \{ v_{2i} \mid i \in \Z_{>0} \}  & Y_2 & :=  \Span \{ v_{2i}^* \mid i \in \Z_{>0} \}.
\end{align*}
By the above theorem, this is equivalent to finding all taut couples $\F$, $\G$ so that the given subspaces provide vector space complements for the pairs in $D$.

Since $Y_1 \oplus Y_2 = V_*$, the semiclosed generalized flag $\G$ in $V_*$ must be either
$0 \subset Y_1 \subset V_*$ or  $0 \subset Y_2 \subset V_*$.  Then $\F$ must be a refinement of the generalized flag $\{ G^\perp \mid G \in \G \}$; that is, $\F$ is a refinement of $0 \subset X_2 \subset V$ or  $0 \subset X_1 \subset V$.  In either case, it is necessary to insert $X_1 \oplus X_2$ into $\F$ in order to have the given subspaces $X_1$ and $X_2$ be vector space complements for the pairs in $D$.  Thus the following is a complete list of the taut couples as desired:
\begin{align*}
0  \subset X_2  & \subset  X_1 \oplus X_2 \subset V \\
V_*  \supset Y_1 & \supset  0  
\end{align*}
and 
\begin{align*}
0 \subset X_1  & \subset  X_1 \oplus X_2 \subset V \\
V_* \supset Y_2 & \supset  0.  
\end{align*}
Note that the subspace $X_1 \oplus X_2$ appearing in both of the above taut couples has codimension $1$ in $V$; nevertheless $(X_1 \oplus X_2)^{\perp \perp} = V$.
  
Let $\l := \bigoplus_{i \in I} \s_i$ for some commuting standard special linear subalgebras of $\g$.  By definition $\s_i = \sl(X_i , Y_i)$ for some subspaces $X_i \subset V$ and $Y_i \subset V_*$.  The maximal trivial $\l$-submodule of $V$ is $(\bigoplus_{i \in I} Y_i)^\perp$, and $\l \cdot V = \bigoplus_{i \in I} X_i$.  Therefore the socle of $V$ as an $\l$-module (that is, the direct sum of all simple $\l$-submodules of $V$) is 
$$\bigoplus_{i \in I} X_i \oplus (\bigoplus_{i \in I} Y_i)^\perp,$$ and each nontrivial simple module in the socle of $V$ has multiplicity $1$.  This shows that each subspace $X_i$ for $i \in I$ is determined by $\l$, and one can recover similarly the subspaces $Y_i$ as the nontrivial simple submodules of $V_*$. 
This enables us to strengthen the above theorem as follows.  Again, let $\p$ be a parabolic subalgebra of $\g$, with the associated taut couple $\F$, $\G$.  The map which takes the subspaces $X_\gamma$, $Y_\gamma$ to the subalgebra $\bigoplus_{\gamma \in D} \sl(X_\gamma, Y_\gamma)$ is a bijection from the sets of subspaces $X_\gamma \subset V$ and $Y_\gamma \subset V_*$ for $\gamma \in D$ such that \eqref{comps} holds and
 $\langle X_\gamma , Y_\eta \rangle = 0 \textrm{ for } \gamma \neq \eta$ to the Levi components of $\p$.

Yet another restatement of Theorem~\ref{iffgl} is in order.  Let $\p$ be a parabolic subalgebra of $\g$, with the associated taut couple $\F$, $\G$.  Then $\l \subset \p$ is a Levi component of $\p$ if and only if the following conditions hold:
\begin{itemize}
\item The $\l$-modules $F''_\gamma / F'_\gamma$ and $G''_\gamma / G'_\gamma$ are simple for all $\gamma \in D$;
\item $\l \cong \bigoplus_i \sl_{n_i}$ for some $n_i \in  \Z_{\geq 2} \cup \{ \infty \}$;
\item  There is a unique nontrivial simple $\s$-submodule of $V$ for each simple direct summand $\s$ of $\l$;
\item For each finite-dimensional simple direct summand $\s$ of $\l$, the nontrivial simple $\s$-submodule of $V$ is isomorphic to the natural or conatural $\s$-module.
\end{itemize}
 The last condition is automatic for the infinite-dimensional simple direct summands of $\l$, as shown in \cite{DimitrovP2}.

\begin{cor}
Not every maximal locally semisimple subalgebra of a finitary Lie algebra is a Levi component. 
\end{cor}
 
 \begin{proof}
We will show that the finitary Lie algebra in question can be chosen to be a parabolic subalgebra of $\gl_\infty$.

Let $V$ and $V_*$ be vector spaces with
 bases $\{ \ldots, v_{-2}, v_{-1}, v_0, v_1, v_2, \ldots \}$ and  $\{ \ldots, v^*_{-2}, v^*_{-1}, v^*_0, v^*_1, v^*_2, \ldots \}$, respectively, and let the pairing be such that these are dual bases. 
  Let
\begin{align*}
X_1 & :=  \Span \{v_1, v_2, v_3, \ldots \}
& Y_1 & :=  \Span \{ v_0^* + v_1*, v_0^* + v_2^*, v_0^* + v_3^* , \ldots \}  
 \\
X_2 & :=  \Span \{v_{-1} , v_{-2}, v_{3}, \ldots \} 
& Y_2 & :=  \Span \{v_{-1}^* , v_{-2}^*, v_{3}^*, \ldots \}.
\end{align*}
Let $\p$ be the stabilizer of $X_1$.  Then $\p$ is the self-normalizing parabolic subalgebra of $\gl(V,V_*)$ corresponding to the taut couple
\begin{align*}
0 & \subset X_1 \subset V \\
V_* & \supset  (X_1)^\perp \supset  0.
\end{align*}  
 Theorem~\ref{iffgl} gives that  $\sl(X_1,Y_1) \oplus \sl(X_2,Y_2)$ is not a Levi component of $\p$, since $$X_1 \oplus X_2 \subsetneq V.$$ 
  We claim that $\sl(X_1,Y_1) \oplus \sl(X_2,Y_2)$ is nevertheless a maximal locally semisimple subalgebra of $\p$. 
 
 To see this, let $\k$ be any maximal locally semisimple subalgebra of $\p$ containing $\sl(X_1, Y_1) \oplus \sl(X_2,Y_2)$.  As one may check from the proof, Lemma~\ref{notdiagonal} can be applied not just to Levi components but also to maximal locally semisimple subalgebras of $\p$, and it implies that $\k$ is a direct sum of standard special linear subalgebras.
Thus $\k$ must have direct summands $\k_1$ and $\k_2$ with $\sl(X_i,Y_i) \subset \k_i = \sl (\tilde{X_i}, \tilde{Y_i})$ for some subspaces $\tilde{X_i} \subset V$ and $\tilde{Y_i} \subset V_*$.  Then one has necessarily that $X_i \subset \tilde{X_i}$ and $Y_i \subset \tilde{Y_i}$.  Note that $\k_1$ and $\k_2$ must be distinct, as $\sl(X_1 \oplus X_2, Y_1 \oplus Y_2)$ is not contained in $\p$.  Hence $[\k_1 , \k_2] = 0$, which implies that $\langle \tilde{X_1}, \tilde{Y_2} \rangle = \langle \tilde{X_2}, \tilde{Y_1} \rangle = 0$.

Consider that $X_2 \subset \tilde{X_2} \subset (Y_1)^\perp = X_2$.  Hence $\tilde{X_2} = X_2$.  Now $$Y_2 \subset \tilde{Y_2} \subset (X_1)^\perp = Y_2 \oplus \C v_0^*.$$  Since the restriction of the pairing on $V \times V_*$ to $\tilde{X_2} \times \tilde{Y_2}$ is nondegenerate, we conclude that $\tilde{Y_2} = Y_2$; that is, $\k_2 = \sl(X_2,Y_2)$.

Similarly, one sees that $X_1 \subsetneq \tilde{X_1}$ implies $\tilde{X_1} = (Y_2)^\perp = X_1 \oplus \C v_0$, while $Y_1 \subsetneq \tilde{Y_1}$ implies $\tilde{Y_1} = (X_2)^\perp = Y_1 \oplus \C v_0^*$.  The only larger potential direct summand $\k_1$ to consider based on nondegeneracy considerations is $\sl (X_1 \oplus \C v_0, Y_1 \oplus \C v_0^*)$, but this does not stabilize $X_1$.  Thus $\k = \sl(X_1,Y_1) \oplus \sl(X_2,Y_2)$, and we have shown that $\sl(X_1,Y_1) \oplus \sl(X_2,Y_2)$ is a maximal locally semisimple subalgebra of $\p$. 
\end{proof}
 
\section{Some further corollaries} \label{corollaries}

We continue to take $\g$ to be $\sl(V,V_*)$ or $\gl(V,V_*)$.  Of the corollaries we present to Theorem~\ref{iffgl}, the first two in particular are useful when computing explicitly all parabolic subalgebras with a given Levi component.

\begin{cor} \label{useful}
Fix a subalgebra $\l = \bigoplus_{i \in I} \sl(X_i,Y_i)$, where $X_i \subset V$ and $Y_i \subset V_*$.  Assume that $\dim X_i \geq 2$ for all $i \in I$, and that $I$ is an ordered set.  Let $U_i \subset V$ be subspaces such that $U_i \oplus X_i \subset U_j$ for all $i < j$ and
 $$U_i = ((U_i \oplus X_i)^\perp \oplus Y_i)^\perp$$ 
 for each $i \in I$. 
 
Let $\F$ be a semiclosed generalized flag maximal among the semiclosed generalized flags in $V$ in which $U_i \subset U_i \oplus X_i$ is an immediate predecessor-successor pair for all $i \in I$. Then there is a unique semiclosed generalized flag $\G$ in $V_*$ such that $\F$, $\G$ form a taut couple and $\l$ is a Levi component of the self-normalizing parabolic subalgebra $\St_\F \cap \St_\G$.
\end{cor}

\begin{proof}
Let $\F$ be maximal among the semiclosed generalized flags in $V$ with immediate predecessor-successor pairs $U_i \subset U_i \oplus X_i$ for all $i \in I$.  Let $\G$ be any semiclosed generalized flag such that $\F$, $\G$ form a taut couple, and $\l$ is a Levi component of $\St_\F \cap \St_\G$.
By Theorem~\ref{iffgl},
for each $i \in I$  there is an immediate predecessor-successor pair $(U_i \oplus X_i)^\perp \subset (U_i \oplus X_i)^\perp \oplus Y_i$ in $\G$.   
By \cite[Proposition 3.3]{DP} each closed subspace in $\G$ is the union of some set of subspaces of the form $F^\perp$ for $F \in \F$, and each closed subspace in $\F$ is the union of some set of subspaces of the form $G^\perp$ for $G \in \G$.  Therefore $\G$ if it exists must have the same stabilizer as the set
\begin{equation*}
\{ F^\perp \mid F \in \F \} \cup \{ (U_i \oplus X_i)^\perp \oplus Y_i \mid i \in I \}.
\end{equation*}

We show that the above set is totally ordered by inclusion, and $(U_i \oplus X_i)^\perp$ is the immediate predecessor of $(U_i \oplus X_i)^\perp \oplus Y_i$ for each $i \in I$.   Indeed, consider that for each $i$ there are no subspaces of the form $F^\perp$ for $F \in \F$ properly between $(U_i \oplus X_i)^\perp$ and $(U_i)^\perp$, since $U_i \subset U_i \oplus X_i$ is an immediate predecessor-successor pair in $\F$.  Furthermore, one has $(U_i \oplus X_i)^\perp \subset (U_i \oplus X_i)^\perp \oplus Y_i \subset (U_i)^\perp$ because of the identity $U_i = ((U_i \oplus X_i)^\perp \oplus Y_i)^\perp$.  
Proposition~\ref{genf} (\ref{uniquesemic}) gives the existence of a unique semiclosed generalized flag $\G$ with the same stabilizer as the above set.  Then $\F$, $\G$ form a taut couple by construction, and Theorem~\ref{iffgl} implies that $\l$ is a Levi component of the self-normalizing parabolic subalgebra $\St_\F \cap \St_\G$.
\end{proof}

The above corollary enables us to determine a self-normalizing parabolic subalgebra with a prescribed Levi component using only subspaces of $V$.  The corollary below shows that any self-normalizing parabolic subalgebra of $\g$ can be so described.

\begin{cor} \label{fullcor}
Let $\p$ be a parabolic subalgebra of $\g$, with the associated taut couple $\F$, $\G$.  Suppose the subalgebra $\bigoplus_{i \in I} \sl(X_i,Y_i)$ is a Levi component of $\p$, where $X_i \subset V$ and $Y_i \subset V_*$.  Assume $\dim X_i \geq 2$ for all $i \in I$.    

Then there exist subspaces $U_i \subset V$ for $i \in I$ with
$$U_i = ((U_i \oplus X_i)^\perp \oplus Y_i )^\perp$$
such that $\F$ is maximal among the semiclosed generalized flags having immediate predecessor-successor pairs
$U_i \subset U_i \oplus X_i$ for all $i \in I$.
\end{cor}

\begin{proof}
Fix $i \in I$.  By Theorem~\ref{iffgl}, there exists some $\gamma \in D$ such that $F''_\gamma = F'_\gamma \oplus X_i$ and $G''_\gamma = G'_\gamma \oplus Y_i$.  Then take $U_i := F'_\gamma$.  Because $F'_\gamma = (G''_\gamma)^\perp$, and $G'_\gamma = (F''_\gamma)^\perp$, we see that $((U_i \oplus X_i)^\perp \oplus Y_i)^\perp = ((F''_\gamma)^\perp \oplus Y_i)^\perp = (G'_\gamma \oplus Y_i)^\perp = (G''_\gamma)^\perp = F'_\gamma = U_i$.

The maximality of $\F$ as stated follows from Theorem~\ref{iffgl}, since all other immediate predecessor-successor pairs of $\F$ (i.e.\ for $\alpha \notin D$) have $\dim F''_\alpha / F'_\alpha = 1$ or $(F'_\alpha)^{\perp \perp} = F''_\alpha$, and in either case admit no further refinement.
\end{proof}

The subspaces $U_i$ of the above two corollaries already made an appearance in the proof of Theorem~\ref{firstcharacterization}.  There we constructed a parabolic subalgebra $\St_\F \cap \St_\G$ as in Corollary~\ref{useful} by choosing $U_i := ((\bigoplus_{k \leq i} X_k)^\perp \oplus Y_i)^\perp$ for all $i \in I$.  If $\p$ is any such parabolic subalgebra of which, using the notation of Corollary~\ref{useful}, the subalgebra $\bigoplus_{i \in I} \sl(X_i, Y_i)$ is a Levi component, then one may check that under the induced order on $I$
$$((\bigoplus_{k \leq i} X_k)^\perp \oplus Y_i)^\perp \subset U_i \subset (\bigoplus_{k \geq i} Y_k)^\perp$$
for each $i \in I$.
We claim that it is also possible in general to construct such parabolic subalgebras by taking $U_i : = (\bigoplus_{k \geq i} Y_k)^\perp$ for all $i \in I$.  Indeed, in this case one may verify the property $U_i = ((U_i \oplus X_i)^\perp \oplus Y_i)^\perp$ using Lemma~\ref{tripleperp} (taking $T = (\bigoplus_{k > i} Y_k)^\perp$, with $X=X_i$ and $Y=Y_i$).  That is, the largest possible subspaces $U_i$ also satisfy the hypotheses of Corollary~\ref{useful}.

Let $\l$ be a subalgebra of $\g$.  The next corollary shows that the parabolic subalgebras of which $\l$ is a Levi component can be distinguished using only a single semiclosed generalized flag in $V$.

\begin{cor} \label{onesided}
Let $\p_1$ and $\p_2$ be parabolic subalgebras of $\g$, with the associated taut couples $\F_1$, $\G_1$ and $\F_2$, $\G_2$, respectively.  Suppose $\p_1$ and $\p_2$ have a Levi component in common.   Then $\F_1 = \F_2$ implies $\G_1 = \G_2$.
\end{cor}

\begin{proof}
By Theorem~\ref{iffgl}, the common Levi component of $\p_1$ and $\p_2$ is of the form $\bigoplus_{\gamma \in D} \sl (X_\gamma, Y_\gamma)$ for some subspaces $X_\gamma \subset V$ and $Y_\gamma \subset V_*$.    As seen in Corollary~\ref{fullcor}, the generalized flag $\F_1 = \F_2$ is maximal among the semiclosed generalized flags in $V$ having $F'_\gamma \subset F'_\gamma \oplus X_\gamma$ as immediate predecessor-successor pairs for all $\gamma \in D$.  Evidently $F'_\gamma \oplus X_\gamma \subset F'_\eta$ for all $\gamma < \eta$ because $\F$ is a generalized flag; the property $F'_\gamma = ((F'_\gamma \oplus X_\gamma)^\perp \oplus Y_\gamma)^\perp$ was shown in the proof of Corollary~\ref{fullcor}.  Thus the uniqueness claim of Corollary~\ref{useful} yields $\G_1 = \G_2$.
\end{proof}

Consider the special case that a parabolic subalgebra $\p$ of $\g$ has $0$ as a maximal locally semisimple subalgebra.  Then $\p$ is a \emph{Borel} subalgebra (that is, a maximal locally solvable subalgebra) of $\g$.  Corollary~\ref{onesided} implies in this case that the associated taut couple of $\p$ is determined by $\F$, the part of the taut couple in $V$.  Since maximal locally solvable subalgebras are minimal parabolic subalgebras, trace conditions are not relevant in this case.  Hence a Borel subalgebra of $\g$ is determined by a single (maximal) semiclosed generalized flag in $V$, as was proved in \cite{DimitrovP1}.

\section{Counting parabolic subalgebras with given Levi component} \label{counting}

In this section we address the question of how many parabolic subalgebras of $\g$ have a given locally semisimple subalgebra $\l$ as a Levi component.  If $\l$ is a Levi component of a parabolic subalgebra $\p$ of a finitary Lie algebra, then $\l$ is also a Levi component of $\p_+$.  Recall that $\p_+$ is a self-normalizing parabolic subalgebra, and $\p$ is defined by trace conditions on $\p_+$.
Therefore we will usually consider first the self-normalizing parabolic subalgebras of $\g$ of which $\l$ is a Levi component.

Fix for $i \in I$ commuting standard special linear subalgebras $\s_i \subset \g$.   When $|I| = n < \infty$, Theorem~\ref{firstcharacterization} implies that there are at least $n!$ self-normalizing parabolic subalgebras of $\g$ having $\bigoplus_i \s_i$ as a Levi component; similarly there are uncountably many such parabolic subalgebras when $I$ is a countable set.  With Theorem~\ref{nblocks} we find criteria for this number to be finite, and we also give an upper bound for this number when it is finite.

\begin{theorem}\label{nblocks}
Fix $\l = \sl(X_1,Y_1) \oplus \cdots \oplus \sl(X_n,Y_n) \subset \g$ for some subspaces $X_i \subset V$ and $Y_i \subset V_*$ with $\dim X_i \geq 2$ for all $i$.  The number of self-normalizing parabolic subalgebras of $\gl(V,V_*)$ with $\l$ as a Levi component is finite if and only if
$$\dim (\bigoplus_{i \notin J} Y_i)^\perp / (\bigoplus_{j \in J} X_j)^{\perp \perp} \leq 1$$
for all subsets $J \subset \{ 1, 2, \ldots, n \}$.  When finite, this number is at most $3 \cdot 2^{n-2} \cdot n!$ for $n \geq 2$, and at most $2$ for $n = 1$; it is uncountable when infinite.
\end{theorem}

\begin{proof}
Suppose first that there exists a subset $J \subset \{ 1, 2, \ldots, n \}$ for which
$$\dim (\bigoplus_{i \notin J} Y_i)^\perp / (\bigoplus_{j \in J} X_j)^{\perp \perp} > 1.$$
Without loss of generality, suppose $J = \{ 1, 2, \ldots, k \}$.  We define
$$U_j =
\begin{cases}
(X_1 \oplus X_2 \oplus \cdots \oplus X_j)^{\perp \perp} \cap Y_j^\perp & \textrm{if } 1 \leq j \leq k \\
(Y_j \oplus Y_{j+1} \oplus \ldots \oplus Y_n)^\perp & \textrm{if } k < j \leq n.
\end{cases}
$$
One may check that $U_1 \subset U_2 \subset \cdots \subset U_n$.  
As described in the paragraph after Corollary~\ref{fullcor}, one has $U_j = ((U_j \oplus X_j)^\perp \oplus Y_j)^\perp$ for $j = 1, \ldots n$.  

Let $\F_0$ be the semiclosed generalized flag
$$0 \subset \cdots \subset U_i \subset U_i \oplus X_i \subset (U_i \oplus X_i)^{\perp \perp} \subset U_{i+1} \subset \cdots \subset V.$$
By Corollary~\ref{useful}, any semiclosed generalized flag $\F$ maximal among the refinements of $\F_0$ retaining the immediacy of the pairs $U_i \subset U_i \oplus X_i$ for $i = 1, \ldots n$ determines a self-normalizing parabolic subalgebra of which $\l$ is a Levi component.

Consider the following portion of $\F_0$:
$$U_k \subset U_k \oplus X_k \subset (U_k \oplus X_k)^{\perp \perp} \subset U_{k+1}.$$
Since $X_1 \oplus X_2 \oplus \cdots \oplus X_{k-1} \subset U_k \subset (X_1 \oplus X_2 \oplus \cdots \oplus X_k)^{\perp \perp}$, we see that $$(U_k \oplus X_k)^{\perp \perp} = (X_1 \oplus X_2 \oplus \cdots \oplus X_k)^{\perp \perp}.$$  We have assumed that $(X_1 \oplus X_2 \oplus \cdots \oplus X_k)^{\perp \perp}$ has codimension at least $2$ in $(Y_{k+1} \oplus Y_{k+2} \oplus \cdots \oplus Y_n)^\perp = U_{k+1}$.  Therefore there are uncountably many closed subspaces between them, and any such closed subspace can appear in a refinement $\F$, $\G$ as described above.  Since different taut couples yield different self-normalizing parabolic subalgebras \cite[Proposition 3.8]{DP}, we conclude that there are uncountably many self-normalizing parabolic subalgebras with $\l$ as a Levi component.

Now suppose that $$\dim (\bigoplus_{i \notin J} Y_i)^\perp / (\bigoplus_{j \in J} X_j)^{\perp \perp} \leq 1$$
for all subsets $J \subset \{ 1, 2, \ldots, n \}$.  We show first that there are at most $2^n \cdot n!$ self-normalizing parabolic subalgebra of $\gl(V,V_*)$ of which $\l$ is a Levi component.  Fix such a parabolic subalgebra $\p$, and denote the associate taut couple by $\F$, $\G$.  By Corollary~\ref{fullcor}, there exist subspaces $U_i$ for $i = 1, \ldots n$ totally ordered by inclusion with the properties listed there.  Without loss of generality, let us assume that $U_1 \subset U_2 \subset \cdots \subset U_n$.  This reindexing produces the factor of $n!$.

Then $\F$ is related to the semiclosed generalized flag
$$0 \subset U_1 \subset \cdots \subset U_i \subset U_i \oplus X_i \subset (U_i \oplus X_i)^{\perp \perp}  \subset U_{i+1} \subset \cdots \subset (U_n \oplus X_n)^{\perp \perp} \subset V,$$
by maximally refining those pairs of the form $(U_i \oplus X_i)^{\perp \perp} \subset U_{i+1}$ for $i = 0, 1, \ldots n$, where we use the notation $U_0 = X_0 = 0$ and $U_{n+1} = V$. 

For $i = 0, \ldots , n$, we have
$$(X_1 \oplus \cdots \oplus X_i)^{\perp \perp} \subset (U_i \oplus X_i)^{\perp \perp} \subset U_{i+1} \subset (Y_{i+1} \oplus \cdots \oplus Y_n)^\perp.$$
By hypothesis $\dim (Y_{i+1} \oplus \cdots \oplus Y_n)^\perp / (X_1 \oplus \cdots \oplus X_i)^{\perp \perp} \leq 1$ for each $i$, hence there are at most two possibilities for each $U_i$.  Thus there at most $2^n$ possible choices of $U_1$, \ldots, $U_n$.  Furthermore, since the pairs $(U_i \oplus X_i)^{\perp \perp} \subset U_{i+1}$ for each $i$ have codimension at most $1$, no further refinement of them is possible.  Hence $\F$ equals the semiclosed generalized flag given above, and by Corollary~\ref{useful} $\p$ is determined by the choice of $U_1$, \ldots, $U_n$.  This shows that the number of self-normalizing parabolic subalgebras with $\l$ as a Levi component is at most $2^n \cdot n!$.  This completes the proof of the if and only if statement, as well as the estimate that the number is question is uncountable when infinite and at most $2^n \cdot n!$ when finite.

Now assume that $n \geq 2$, and that the number of self-normalizing parabolic subalgebras with $\l$ as a Levi component is finite.  To prove the stated upper bound of $3 \cdot 2^{n-2} \cdot n!$, we show that there are at most three possible combinations of $U_1$ and $U_2$.  

Assume therefore that
$$(X_1)^{\perp \perp} \cap (Y_1)^\perp \subsetneq (Y_1 \oplus Y_2 \oplus \cdots \oplus Y_n)^\perp \subset (X_1 \oplus X_2)^{\perp \perp} \cap (Y_2)^\perp.$$
(If this assumption does not hold, it is clear that there are at most three possibilities for $U_1$ and $U_2$.)  It suffices to show that
$$(X_1 \oplus X_2)^{\perp \perp} \cap (Y_2)^\perp = (Y_2 \oplus \cdots \oplus Y_n)^\perp.$$ 

We have already shown that $\dim (Y_1 \oplus Y_2 \oplus \cdots \oplus Y_n)^\perp \leq 1$.  The condition $$(X_1)^{\perp \perp} \cap (Y_1)^\perp \subsetneq (Y_1 \oplus Y_2 \oplus \cdots \oplus Y_n)^\perp$$ therefore implies that $(X_1)^{\perp \perp} \cap (Y_1)^\perp = 0$ and  $\dim (Y_1 \oplus Y_2 \oplus \cdots \oplus Y_n)^\perp = 1$.

We have also assumed that $(Y_1 \oplus Y_2 \oplus \cdots \oplus Y_n)^\perp \subset (X_1 \oplus X_2)^{\perp \perp} \cap (Y_2)^\perp$, so indeed $$(X_1)^{\perp \perp} + (Y_1 \oplus Y_2 \oplus \cdots \oplus Y_n)^\perp \subset (X_1 \oplus X_2)^{\perp \perp} \cap (Y_2)^\perp.$$  Since $(X_1)^{\perp \perp}  \cap (Y_1 \oplus Y_2 \oplus \cdots \oplus Y_n)^\perp \subset (X_1)^{\perp \perp} \cap (Y_1)^\perp = 0$, we have the direct sum 
$$(X_1)^{\perp \perp} \oplus (Y_1 \oplus Y_2 \oplus \cdots \oplus Y_n)^\perp \subset (X_1 \oplus X_2)^{\perp \perp} \cap (Y_2)^\perp \subset (Y_2 \oplus \cdots \oplus Y_n)^\perp.$$ 
We have already proved that the codimension of $(X_1)^{\perp \perp}$ in  $(Y_2 \oplus \cdots \oplus Y_n)^\perp$ is at most $1$; hence
$$(X_1)^{\perp \perp} \oplus (Y_1 \oplus Y_2 \oplus \cdots \oplus Y_n)^\perp = (X_1 \oplus X_2)^{\perp \perp} \cap (Y_2)^\perp = (Y_2 \oplus \cdots \oplus Y_n)^\perp,$$
and we are done. 
\end{proof}

Let us consider the finite numbers obtained as the number of self-normalizing parabolic subalgebras with a given subalgebra as a Levi component, subject to the restriction that the given subalgebra has $n$ simple direct summands.  In the case $n \geq 2$, Theorem~\ref{nblocks} says that the number of self-normalizing parabolic subalgebras with Levi component $\s_1 \oplus \cdots \oplus \s_n$ is at most $3 \cdot 2^{n-2} \cdot n!$ if it is finite.   (One can imagine that this upper bound is typically not sharp.)  Considerations completely analogous to the finite-dimensional case show that the maximum of this set of finite numbers is at least $(n+1)!$.   Nevertheless, the very last example in this section shows that, unlike in the finite-dimensional case, $(n+1)!$ is not an upper bound when $n=5$.

As an illustration of Theorem~\ref{nblocks}, one might ask how many parabolic subalgebras of $\gl(V,V_*)$ have Levi component $\sl(X_1 , Y_1) \oplus \sl(X_2 , Y_2)$, where 
\begin{align*}
X_1 & :=  \Span \{ v_1 + v_{2i - 1} \mid i \geq 2 \} & Y_1 & :=   \Span \{ v_{2i - 1}^* \mid i \geq 2  \} \\
X_2 & :=  \Span \{ v_1 + v_{2i} \mid i \geq 2 \} & Y_2 & :=  \Span \{  v_2^* + v_{2i}^* \mid i \geq 2  \},
\end{align*}
and $V$ and $V_*$ are vector spaces with dual bases $\{ v_i \mid i \in \Z_{> 0} \}$ and $\{ v_i^* \mid i \in \Z_{> 0} \}$.  
To answer the question, we compute the four quotient spaces in the hypotheses of Theorem~\ref{nblocks}:
\begin{itemize}
\item $(Y_1 \oplus Y_2)^\perp = \C v_1 $,
\item $(Y_1)^\perp / (X_2)^{\perp \perp} = \C v_1 \oplus \Span \{ v_{2i} \mid i \geq 1 \} /  \C v_1 \oplus \Span \{ v_{2i} \mid i \geq 2 \} $,
\item $(Y_2)^\perp / (X_1)^{\perp \perp} =  \Span \{ v_{2i -1} \mid i \geq 1 \} / \Span \{ v_{2i - 1} \mid i \geq 1 \}$,
\item $V / (X_1 \oplus X_2)^{\perp \perp} = V / \Span \{ v_i \mid i \neq 2 \}$.
\end{itemize}
Since all the above have dimension no greater than $1$, Theorem~\ref{nblocks} implies that only a finite number of self-normalizing parabolic subalgebras of $\gl(V,V_*)$ have $\sl(X_1 , Y_1) \oplus \sl(X_2 , Y_2)$ as a Levi component.  Indeed, one may check that there are precisely three such self-normalizing parabolic subalgebras.  Explicitly, they are the stabilizers of the following three taut couples.  The order $1 < 2$ gives the taut couple
\begin{align*}
0 & \subset \C v_1  \subset  \C v_1 \oplus X_1 \subset \C v_1 \oplus X_1 \oplus X_2 \subset V \\
V_* &  \supset (v_1)^\perp  \supset Y_2 \oplus \C v_2^* \supset \C v_2^* \supset  0 , \\
\intertext{while the other order gives the two taut couples}
0 & \subset \C v_1  \subset  \C v_1 \oplus X_2 \subset \C v_1 \oplus X_2 \oplus X_1 \subset V \\
V_* &  \supset (v_1)^\perp  \supset  \C v_2^* \oplus Y_1 \supset \C v_2^* \supset  0 \\
\intertext{and}
0 & \subset \C v_1  \subset  \C v_1 \oplus X_2 \subset \C v_1 \oplus X_2 \oplus \C v_2 \subset V \\
V_* &  \supset (v_1)^\perp  \supset Y_1 \oplus \C v_2^* \supset Y_1 \supset  0.  
\end{align*}

Corollary~\ref{1block} addresses the special case $n=1$ of Theorem~\ref{nblocks}.

\begin{cor} \label{1block}
Fix subspaces $X \subset V$ and $Y \subset V_*$ such that $\langle \cdot , \cdot \rangle |_{X \times Y}$ is nondegenerate.  The number of self-normalizing parabolic subalgebras of $\g$ with $\sl(X,Y)$ as a Levi component is finite if and only if
$$\dim X^\perp \leq 1 \textrm{ and } \dim Y^\perp \leq 1.$$
When it is finite, this number is $1$ if $\langle Y^\perp , X^\perp \rangle = 0$ and $2$ if $\langle Y^\perp , X^\perp \rangle \neq 0$.
\end{cor}

\begin{proof}
Theorem~\ref{nblocks} implies that the number of such parabolic subalgebras is finite if and only if 
$\dim Y^\perp / 0 \leq 1$ and $\dim V / X^{\perp \perp} \leq 1$.  Since $\dim V / X^{\perp \perp} = \dim X^\perp$, the if and only if statement is clear.

Now suppose $\dim X^\perp \leq 1$ and $\dim Y^\perp \leq 1$.  No further refinement of the semiclosed generalized flag $0 \subset U \subset U \oplus X \subset V$ is possible for any subspace $U \subset V$ such that $U = ((U \oplus X)^\perp \oplus Y)^\perp$.  Therefore by Corollaries~\ref{useful} and \ref{fullcor}, the parabolic subalgebras in question are in bijection with the subspaces $U \subset V$ such that $U = ((U \oplus X)^\perp \oplus Y)^\perp$.  

Let $U \subset V$ be such that $U = ((U \oplus X)^\perp \oplus Y)^\perp$.  It follows that $X^{\perp \perp} \cap Y^\perp \subset U \subset Y^\perp$. 
 If $\langle Y^\perp , X^\perp \rangle = 0$, then $X^{\perp \perp} \cap Y^\perp = Y^\perp$, so $U = Y^\perp$, i.e.\ there is exactly one such parabolic subalgebra.  If $\langle Y^\perp, X^\perp \rangle \neq 0$, then $X^{\perp \perp} \cap Y^\perp=0$ and $Y^\perp \neq 0$, so $U=0$ and $U = Y^\perp$ are the two possibilities.
\end{proof}

The following example gives an illustration of Corollary~\ref{1block}.  
We will find all parabolic subalgebras of $\gl(V,V_*)$ with the simple subalgebra $\sl(X,Y)$ as a Levi component, where 
\begin{align*}
X & :=  \Span \{ v_1 + v_i \mid i \geq 2 \}  & Y & :=  \Span \{ v_i^* \mid i \geq 2 \},
\end{align*}
and $\{ v_i \mid i \in \Z_{> 0} \}$ and $\{ v_i^* \mid i \in \Z_{> 0} \}$ are dual bases of the vector spaces $V$ and $V_*$.  

We compute $X^\perp = 0$ and $Y^\perp = \C v_1$.  By Corollary~\ref{1block} there is exactly $1$ self-normalizing parabolic subalgebra of $\gl(V,V_*)$ having $\sl(X,Y)$ as a Levi component.  It is the stabilizer of the taut couple
\begin{eqnarray*}
0 \subset & \C v_1 & \subset V\\
V_* \supset & Y & \supset 0, 
\end{eqnarray*}
which by a computation is nothing but $$\p_+ = \big( (v_1^*)^\perp \otimes v_1^* \big) \subsetplus \big( ( \C v_1 \otimes v_1^*) \oplus ((v_1^*)^\perp \otimes Y) \big).$$  Thus every parabolic subalgebra $\p$ of $\gl(V,V_*)$ with $\sl(X,Y)$ as a Levi component is defined by trace conditions on the one infinite-dimensional block of a locally reductive part of $\p_+$.  There are precisely two such parabolic subalgebras, namely $\p_+$ and $\p_- = \big( (v_1^*)^\perp \otimes v_1^* \big) \subsetplus \big( ( \C v_1 \otimes v_1^*) \oplus  \sl((v_1^*)^\perp, Y) \big)$.

Similarly, there are precisely two parabolic subalgebras of $\sl(V,V_*)$ with Levi component $\sl(X,Y)$.  They are the traceless parts of the two parabolic subalgebras in the previous paragraph.

It is not hard to see now that it is impossible to extend $\sl(X,Y)$ to a locally reductive part of $\p_+$.  Any locally reductive part of $\p_+$ must be isomorphic to the locally reductive part $(\C v_1 \otimes v_1^*) \oplus ((v_1^*)^\perp \otimes Y)$, hence it must have two commuting blocks.  However, the centralizer of $\sl(X,Y)$ in $\gl(V,V_*)$ is trivial.  Relatedly, one may check that $\gl(X,Y)$ is a maximal locally reductive subalgebra of $\p_+$ which is not a locally reductive part of $\p_+$.

We conclude this section with an example to demonstrate that $(n+1)!$ is not an upper bound for the finite numbers which occur as the number of self-normalizing parabolic subalgebras with prescribed Levi component when $n=5$.
We claim that there are precisely $8 \cdot 5!$ self-normalizing parabolic subalgebras of $\gl(V,V_*)$ with Levi component  $\sl(X_1,Y_1) \oplus \cdots \oplus \sl(X_5,Y_5)$, where  
\begin{align*}
X_k & := \Span \{ v_i \mid i = k \operatorname{mod} 5 \} & Y_k & := \Span \{ v_i^* \mid i = k \operatorname{mod} 5 \},
\end{align*}
in the following notation.   We take $V$ and $V_*$ to be the vector spaces with bases $\{z\} \cup \{w_1, w_2, \ldots, w_{15} \} \cup \{ v_i \mid i \in \Z_{>0} \}$ and $\{ \tilde{z} \} \cup \{ \tilde{w}_1, \tilde{w}_2, \ldots, \tilde{w}_{15}  \} \cup \{ v_i^* \mid i \in \Z_{>0} \}$.  
Let $\langle \cdot , \cdot \rangle : V \times V_* \rightarrow \C$ be the nondegenerate pairing defined by setting
\begin{align*}
\langle v_i , v_j^* \rangle & = \delta_{ij} \\
\langle z , \tilde{z} \rangle & = 0 \\
\langle w_k , \tilde{w}_l \rangle & = 1 \\
\langle v_i, \tilde{z} \rangle = \langle z , v_i^* \rangle & = 0 \\
\langle z , \tilde{w}_k \rangle = \langle w_k , \tilde{z} \rangle & =
\begin{cases}
0 & \textrm{if } 1 \leq k \leq 10 \\
1 & \textrm{if } 11 \leq k \leq 15
\end{cases} \\
\langle v_i, \tilde{w}_k \rangle = \langle w_k , v_i^* \rangle & = 
\begin{cases}
1 & \textrm{if } i \textrm{ is congruent $\operatorname{mod} 40$ to an element of } S_k \\
0 & \textrm{otherwise}
\end{cases}
\end{align*}
for all $i, j \in \Z_{>0}$ and $k, l \in \{1, \ldots, 15 \}$, 
where the sets $S_k$ are the following:
\begin{align*}
S_1 & = \{1,2\} &
S_5 & = \{9, 12\} &
S_9 & = \{ 15, 18 \} & 
S_{13} &= \{29,30,31,32 \} \\
S_2 & = \{3,6\} & 
S_6 & = \{13,14\} &
S_{10} & = \{19, 20\} &
S_{14} & = \{33,34,35,36 \} \\
S_3 & = \{7,8\} & 
S_7 & = \{5, 16\} & 
S_{11} &= \{21,22,23,24 \} &
S_{15} & = \{37,38,29,0 \}. \\
S_4 & = \{ 4, 11\} &
S_8 & = \{10,17\} &
S_{12} &= \{25,26,27,28 \} &&
\end{align*}
 
It suffices to prove that there are $8$ such parabolic subalgebras inducing the usual order $1 < 2 < 3 <4 < 5$, due to symmetry.  By Corollaries~\ref{useful} and \ref{fullcor}, these parabolic subalgebras are in correspondence with sets of subspaces $U_1$, \ldots, $U_5 \subset V$ such that $U_i = ((U_i \oplus X_i)^\perp \oplus Y_i )^\perp$ and $U_i \oplus X_i \subset U_{i+1}$.
Indeed, the parabolic subalgebra associated to $U_1$, \ldots, $U_5$ is the stabilizer of the taut couple
\begin{align*} 
0 \subset U_1 \subset U_1 \oplus X_1  \subset (U_1 \oplus X_1)^{\perp \perp} & \subset  U_2 \subset U_2 \oplus X_2 \subset \cdots \\
\cdots \subset (U_4 \oplus X_4)^{\perp \perp} & \subset U_5 \subset U_5 \oplus X_5 \subset (U_5 \oplus X_5)^{\perp \perp} \subset V \\
V_* \supset (Y_1 \oplus T_1)^{\perp \perp} \supset Y_1 \oplus T_1 \supset  T_1 & \supset (Y_2 \oplus T_2)^{\perp \perp}  \supset \cdots  \\ 
\cdots \supset  Y_4 \oplus T_4 \supset T_4 & \supset (Y_5 \oplus T_5)^{\perp \perp}  \supset Y_5 \oplus T_5 \supset T_5 \supset 0,
\end{align*}
where $T_i := (U_i \oplus X_i)^\perp$.

By the proof of Theorem~\ref{nblocks}, there are at most $2$ possibilities for each $U_i$.  Thus each $U_i$ must be either $\big( \big( \bigoplus_{j \leq i} X_j \big)^\perp \oplus Y_i \big)^\perp$ or $\big( \bigoplus_{j \geq i} Y_j \big)^\perp$.  
This enables us to list all the possibilities:
\begin{align*}
U_1 & =  0 
\textrm{ or } 
\Span\{ z \}   \\
U_2 & = X_1\oplus  \Span\{ z \} \\
U_3   & =  X_1\oplus X_2 \oplus \Span\{ z \} 
\textrm{ or }
X_1\oplus X_2 \oplus \Span\{ z , w_1 \}  \\
 U_4 & =  X_1\oplus X_2 \oplus X_3 \oplus \Span\{ z , w_1 , w_2 , w_3 \} \\
 U_5 & =  X_1\oplus X_2 \oplus X_3 \oplus X_4 \oplus \Span\{ z , w_1 , w_2 , w_3 , w_4 , w_5, w_6 \} \textrm{ or}\\
& \hspace{1em} \hspace{1ex} X_1\oplus X_2 \oplus X_3 \oplus X_4 \oplus \Span\{ z , w_1 , w_2 , w_3 , w_4 , w_5, w_6, w_{11} \}. \end{align*}
Observe that for every combination of choices, the necessary inclusions remain, i.e.\ $U_i \oplus X_i \subset U_{i+1}$.  The listed subspaces all satisfy $U_i = ((U_i \oplus X_i)^\perp \oplus Y_i)^\perp$, as noted immediately after Corollary~\ref{fullcor}.  Hence there are exactly $8$ self-normalizing parabolic subalgebras as desired, arising from the two possibilities each for $U_1$, $U_3$, and $U_5$.

\section{Levi components of parabolic subalgebras of $\so_\infty$ and $\sp_\infty$}

Assume $\g$ is $\so(V)$ or $\sp(V)$.  We omit the proofs, as they are similar to those given above.

\begin{theorem}
Let $\l$ be a subalgebra of $\g$.  There exists a parabolic subalgebra $\p$ of $\g$ such that $\l$ is a Levi component of $\p$ if and only if $\l$ is the direct sum of standard special linear subalgebras of $\g$ and a subalgebra 
$$\k =
\begin{cases}
\so(W) & \textrm{if } \g = \so(V) \\
\sp(W)  & \textrm{if } \g = \sp(V)
\end{cases}$$
for some subspace $W \subset V$ to which the restriction of the bilinear form on $V$ is nondegenerate.

Moreover, given a subalgebra $\l$ for which such parabolic subalgebras exist, one exists that induces an arbitrary order on the standard special linear direct summands of $\l$.
\end{theorem}

\begin{theorem} \label{iffso}
Let $\p$ be a parabolic subalgebra of $\so(V)$, with an associated self-taut generalized flag $\F$.  Let $F$ denote the union of all isotropic subspaces $F''_\alpha$ for $\alpha \in A$, and let $G$ denote the intersection of all coisotropic subspaces $F'_\alpha$ for $\alpha \in A$.

Then $\l$ is a Levi component of $\p$ if and only if 
there exist isotropic subspaces $X_\gamma \subset V$ and $Y_\gamma \subset V$ for each $\gamma \in D$ with
$$F''_\gamma = F'_\gamma \oplus X_\gamma \textrm{ and } G''_\gamma = G'_\gamma \oplus Y_\gamma,$$ 
 as well as a subspace $W$ with
$W=0$ if $\dim G / F \leq 2$ and otherwise
$$G = F \oplus W$$
such that
$$\l = \so(W) \oplus \bigoplus_{\gamma \in D} \Lambda(\sl(X_\gamma, Y_\gamma)).$$
\end{theorem}

\begin{theorem} \label{iffsp}
Let $\p$ be a parabolic subalgebra of $\sp(V)$, with the associated self-taut generalized flag $\F$.  Let $F$ denote the union of all isotropic subspaces $F''_\alpha$ for $\alpha \in A$, and let $G$ denote the intersection of all coisotropic subspaces $F'_\alpha$ for $\alpha \in A$.

Then $\l$ is a Levi component of $\p$ if and only if 
there exist isotropic subspaces $X_\gamma \subset V$ and $Y_\gamma \subset V$ for each $\gamma \in D$ with
$$F''_\gamma = F'_\gamma \oplus X_\gamma \textrm{ and } G''_\gamma = G'_\gamma \oplus Y_\gamma,$$ 
as well as a subspace $W$ with 
$$G = F \oplus W$$
such that
$$\l = \sp(W) \oplus \bigoplus_{\gamma \in D} S(\sl(X_\gamma, Y_\gamma)).$$
\end{theorem}

The parabolic subalgebra in the following example has Levi components isomorphic to $\g$. 
Let $V$ be the vector space with basis $\{ v_i \mid  i \in \Z_{\neq 0} \}$, and let $\langle \cdot , \cdot \rangle : V \times V \rightarrow \C$ be the nondegenerate pairing extending
$$\langle v_i , v_j \rangle =
\begin{cases}
 0 & \textrm{if } i \neq -j \\
1 & \textrm{if }  i = -j > 0
\end{cases}$$
symmetrically (or antisymmetrically). Take $$W := \Span \{ v_1 + v_i \mid i \neq \pm 1 \},$$
and note that the restriction $\langle \cdot , \cdot \rangle |_{W \times W}$ is nondegenerate. 
We will show that there is a unique parabolic subalgebra of $\g$ with $\so(W)$ (resp. $\sp(W)$) as a Levi component.  In order to apply Theorem~\ref{iffso} (resp., Theorem~\ref{iffsp}), we consider self-taut generalized flags $\F$ such that $W$ provides a vector space complement for the single immediate predecessor-successor pair in $D$.

That is, any self-taut generalized flag $\F$ in $V$ such that $\St_\F \cap \g$ has the prescribed Levi component must have an immediate predecessor-successor pair of the form $U \subset U \oplus W$.  Since $W$ is neither isotropic nor coisotropic, $U$ must be isotropic, with $U = (U \oplus W)^\perp$.  This implies $W^{\perp \perp} \cap W^\perp \subset U$.  We compute $W^{\perp \perp} \cap W^\perp = \C v_1$.  As $\dim V / W = 2$, we conclude that $U = \C v_1$, so $\F$ is the self-taut generalized flag
$$0 \subset \C v_1 \subset \C v_1 \oplus W \subset V.$$ 
This yields a single self-normalizing parabolic subalgebra $\p := \St_\F \cap \g$ as desired.  There are no nontrivial trace conditions on $\p$, as there are no $\gl_\infty$ blocks in a locally reductive part; hence $\p$ is the unique parabolic subalgebra of $\g$ with the prescribed Levi component.  

Note that $\so(W)$ (or $\sp(W)$) is a maximal locally reductive subalgebra of $\p$.  On the other hand, any reductive part of $\p$ is isomorphic to $\g \oplus \C$.  As in the example at the end of Section~\ref{counting}, we have found a maximal locally reductive subalgebra of $\p$ which is not a locally reductive part of $\p$.

Now that we have considered the three special cases of $\sl_\infty$, $\so_\infty$, and $\sp_\infty$, the analogous statements hold for locally reductive finitary Lie algebras $\g$.  Then $[\g,\g]$ is locally simple, hence $[\g,\g] = \bigoplus_{i \in I} \s_i$ for some simple finitary Lie algebras $\s_i$.  Let $\l$ be a Levi component of a parabolic subalgebra $\p$ of $\g$.  Then $\p \cap \s_i$ is a parabolic subalgebra of $\s_i$ for each $i \in I$.  It must be the case that $\l = \bigoplus_{i \in I} \l \cap \s_i$, and moreover $\l \cap \s_i$ is a Levi component of $\p \cap \s_i$.  Up to isomorphism the only infinite-dimensional simple finitary Lie algebras are $\sl_\infty$, $\so_\infty$, and $\sp_\infty$ \cite{BaranovStrade}.  Thus the results of this paper are enough to classify Levi components of parabolic subalgebras in this generality.

\bibliographystyle{amsplain}

\begin{thebibliography}{A}
\bibitem[\textbf{B}]{Baranov1} A. Baranov, Complex finitary simple Lie
algebras, Arch. Math. 72 (1999), 101--106.
\bibitem[\textbf{B2}]{Baranov2} A. Baranov, Finitary simple Lie
algebras, J. Algebra 219 (1999), 299--329.
\bibitem[\textbf{BS}]{BaranovStrade} A. Baranov, H. Strade, Finitary Lie algebras, J. Algebra 254 (2002),  173--211.
\bibitem[\textbf{Bo}]{Bourbaki} N. Bourbaki, Groupes et alg\`{e}bres de Lie, Hermann, Paris, 1975.
\bibitem[\textbf{D}]{D} E. Dan-Cohen, Borel subalgebras of root-reductive Lie algebras, J. Lie Theory 18 (2008), 215--241.
\bibitem[\textbf{DP}]{DP} E. Dan-Cohen, I. Penkov, Parabolic and Levi subalgebras of finitary Lie algebras, Internat. Math. Res. Notices 2010, No. 6 (2010), 1062--1101. 
\bibitem[\textbf{DPSn}]{DPSnyder} E. Dan-Cohen, I. Penkov, N. Snyder, Cartan subalgebras of root-reductive Lie algebras, J. Algebra 308 (2007), 583--611.
\bibitem[\textbf{DPW}]{DPWolf} E. Dan-Cohen, I. Penkov, J.A. Wolf, Parabolic subgroups of direct limit Lie groups, Cont. Math. 499 (2009) 47--59.
\bibitem[\textbf{DiP}]{DimitrovP1} I. Dimitrov, I. Penkov, Borel subalgebras of $\gl(\infty)$, Resenhas do Instituto de Matem\'{a}tica e Estat\'{i}stica da Universidade de S\~{a}o Paulo 6 (2004), No. 2/3, 153--163.
\bibitem[\textbf{DiP2}]{DimitrovP2} I. Dimitrov, I. Penkov, Locally semisimple and maximal subalgebras of the finitary Lie algebras $\gl(\infty)$, $\sl(\infty)$, $\so(\infty)$, and $\sp(\infty)$, J. Algebra 322 (2009), 2069--2081.
\bibitem[\textbf{M}]{Mackey} G. Mackey, On infinite dimensional linear spaces, Trans. Amer. Math. Soc. 57 (1945) 155--207.
\bibitem[\textbf{NP}]{NeebP} K.-H. Neeb, I. Penkov, Cartan subalgebras of $\gl_\infty$, Canad. Math. Bull. 46 (2003), 597--616.
\bibitem[\textbf{PS}]{PStrade} I. Penkov, H. Strade, Locally finite Lie algebras with root decomposition, Arch. Math. 80 (2004), 478--485.
\end{thebibliography}

\end{document}